\documentclass[12pt]{article}
\usepackage{amsmath,amssymb,amsthm}

\newcommand{\R}{{\mathbb R}}
\newcommand{\Z}{{\mathbb Z}}
\newcommand{\N}{{\mathbb N}}
\newcommand{\C}{{\mathbb C}}
\newcommand{\Q}{{\mathbb Q}}
\newcommand{\D}{{\mathcal D}}
\newcommand{\T}{{\mathbb T}}
\newcommand{\B}{{\mathcal B}}

\newcommand{\meas}{\mathop{\rm meas}}

\newcommand{\const}{{\rm const}}
\newcommand{\esslim}{\mathop{\rm ess\,lim}}

\newcommand{\sign}{\mathop{\rm sign}}
\newcommand{\supp}{\mathop{\rm supp}}
\renewcommand{\div}{{\rm div}}

\def\Xint#1{\mathchoice
{\XXint\displaystyle\textstyle{#1}}%
{\XXint\textstyle\scriptstyle{#1}}%
{\XXint\scriptstyle\scriptscriptstyle{#1}}%
{\XXint\scriptscriptstyle\scriptscriptstyle{#1}}%
\!\int}
\def\XXint#1#2#3{{\setbox0=\hbox{$#1{#2#3}{\int}$ }
\vcenter{\hbox{$#2#3$ }}\kern-.57\wd0}}

\def\dashint{\Xint-}
\numberwithin{equation}{section}

\theoremstyle{plain}
\newtheorem{theorem}{Theorem}[section]
\newtheorem{lemma}{Lemma}[section]
\newtheorem{proposition}{Proposition}[section]
\newtheorem{corollary}{Corollary}[section]

\theoremstyle{definition}
\newtheorem{definition}{Definition}[section]
\newtheorem{remark}{Remark}[section]
\newtheorem{example}{Example}[section]

\title{On the Cauchy problem for scalar conservation laws in the class of Besicovitch almost periodic functions: global well-posedness and decay property}
\author{E.Yu.~Panov\footnote{Novgorod State University, e-mail: \textbf{Eugeny.Panov@novsu.ru}}}

\begin{document}
\maketitle
\begin{abstract} We study the Cauchy problem for a multidimensional scalar conservation law with merely continuous flux vector in the class of Besicovitch almost periodic functions. The existence and uniqueness of entropy solutions are established. We propose also the necessary and sufficient condition for the decay of almost periodic entropy solutions as time $t\to+\infty$.
\end{abstract}

\section{Introduction}\label{sec1}
In the half-space $\Pi=\R_+\times\R^n$, $\R_+=(0,+\infty)$, we consider the Cauchy problem for a first order multidimensional
conservation law
\begin{equation}\label{1}
u_t+\div_x\varphi(u)=0
\end{equation}
with initial data
\begin{equation}\label{ini}
u(0,x)=u_0(x).
\end{equation}
The flux vector $\varphi(u)$ is supposed to be only continuous: $$\varphi(u)=(\varphi_1(u),\ldots,\varphi_n(u))\in C(\R,\R^n).$$
Assume that $u_0(x)\in L^\infty(\R^n)$. Then the notion of entropy solution of (\ref{1}), (\ref{ini}) in the sense of S.N.~Kruzhkov \cite{Kr} is well-defined.

\begin{definition}\label{def1}
A bounded measurable function $u=u(t,x)\in L^\infty(\Pi)$ is called an entropy solution (e.s. for
short) of (\ref{1}), (\ref{ini}) if for all $k\in\R$
\begin{equation}\label{2}
|u-k|_t+\div_x[\sign(u-k)(\varphi(u)-\varphi(k))]\le 0
\end{equation}
in the sense of distributions on $\Pi$ (in $\D'(\Pi)$);
$$\esslim_{t\to 0} u(t,\cdot)=u_0 \ \mbox{ in } L^1_{loc}(\R^n).$$
\end{definition}

Condition (\ref{2}) means that for all non-negative test functions $f=f(t,x)\in C_0^1(\Pi)$
$$
\int_\Pi [|u-k|f_t+\sign(u-k)(\varphi(u)-\varphi(k))\cdot\nabla_xf]dtdx\ge 0
$$
(here $\cdot$ denotes the inner product in $\R^n$).

It is known that e.s. is always exists (see \cite{KrPa1,PaMV,PaIzv}~) but, in the case under consideration when the flux
functions are merely continuous, this e.s. may be nonunique (see examples in \cite{KrPa1,KrPa2}). Nevertheless, if initial function is periodic (at least in $n-1$ independent directions), the uniqueness holds: an e.s. of (\ref{1}), (\ref{ini}) is unique and space-periodic, see the proof in \cite{PaMax1,PaMax2,PaIzv}.

It is also well-known (see for instance \cite[Proposition~1]{PaIzv}) that initial requirement can be included in the single integral entropy inequality:
for all nonnegative test functions $f=f(t,x)\in C_0^1(\bar\Pi)$, where $\bar\Pi=[0,+\infty)\times\R^n$, for all $k\in\R$
\begin{equation}\label{enint}
\int_{\R^n}|u_0(x)-k|f(0,x)dx+\int_\Pi [|u-k|f_t+\sign(u-k)(\varphi(u)-\varphi(k))\cdot\nabla_xf]dtdx\ge 0,
\end{equation}
that is, $u=u(t,x)$ is an e.s. of (\ref{1}), (\ref{ini}) if and only if $u$ satisfies relation (\ref{enint}).

We will essentially rely on the following mean $L^1$-contraction property. Denote by $C_R$ the cube
$$\{ \ x=(x_1,\ldots,x_n)\in\R^n \ | \ |x|_\infty=\max_{i=1,\ldots,n}|x_i|\le R/2 \ \}, \quad R>0$$ and let
$$N_p(u)=\limsup_{R\to +\infty}\left(R^{-n}\int_{C_R} |u(x)|^pdx\right)^{1/p}, \quad p\ge 1,
$$
be the mean $L^p$-norm of a function $u(x)\in L^\infty(\R^n)$.

\begin{proposition}[mean $L^1$-contraction property] \label{pro1}
Let $u(t,x),v(t,x)\in L^\infty(\Pi)$ be e.s. of (\ref{1}), (\ref{ini}) with initial functions $u_0(x), v_0(x)$, respectively. Then for almost every (a.e.) $t>0$
\begin{equation}\label{m-contr}
N_1(u(t,\cdot)-v(t,\cdot))\le N_1(u_0-v_0).
\end{equation}
\end{proposition}

\begin{proof}
Applying Kruzhkov doubling of variables method, we obtain the relation (see \cite{Kr,PaMV,PaIzv})
\begin{equation}\label{3}
|u-v|_t+\div_x[\sign(u-v)(\varphi(u)-\varphi(v))]\le 0 \ \mbox{ in } \D'(\Pi).
\end{equation}
Let $\rho(s)\in C_0^\infty(\R)$ be a function such that $\supp\rho(s)\subset [0,1]$, $\rho(s)\ge 0$,
${\displaystyle\int_{-\infty}^{+\infty}\rho(s)ds=1}$. We set for $\nu\in\N$
$$
\delta_\nu(s)=\nu\rho(\nu s), \quad\theta_\nu(t)=\int_0^t\delta_\nu(s)ds=\int_0^{\nu t}\rho(s)ds.
$$
Obviously, the sequence $\delta_\nu(s)$ converges as $\nu\to\infty$ to the Dirac $\delta$-measure weakly in $\D'(\R)$ while the sequence $\theta_\nu(t)$ pointwise converges to the Heaviside function. If $t_1>t_0>0$ then the function $\chi_\nu(t)=\theta_\nu(t-t_0)-\theta_\nu(t-t_1)\in C_0^\infty(\R^+)$, $0\le\chi_\nu(t)\le 1$, and the sequence $\chi_\nu(t)$ pointwise converges as ${\nu\to\infty}$ to the indicator function $\chi(t)$ of the interval $(t_0,t_1]$.
We choose the function $g(y)\in C_0^\infty(\R^n)$ with the properties: $0\le g(y)\le 1$, $g(y)\equiv 1$ in the cube $C_1$, $g(y)\equiv 0$ in the complement of the cube $C_k$, $k>1$ (~so that $\supp g(x)\subset C_k$~). Applying relation (\ref{3}) to the test function $f=R^{-n}\chi_\nu(t)g(x/R)$, where $R>0$, we arrive at
\begin{eqnarray}\label{3i}
\int_0^\infty\left(R^{-n}\int_{\R^n} |u(t,x)-v(t,x)| g(x/R)dx\right)(\delta_\nu(t-t_0)-\delta_\nu(t-t_1))dt+ \nonumber\\ R^{-n-1}\int_\Pi \sign(u-v)(\varphi(u)-\varphi(v))\cdot\nabla_y g(x/R)\chi_\nu(t)dtdx\ge 0.
\end{eqnarray}
Define the set
$$
F=\{ \ t>0 \ | \ (t,x) \ \mbox{ is a Lebesgue point of } |u(t,x)-v(t,x)| \mbox{ for a.e. } x\in\R^n \ \}.
$$
By the technical Lemma~\ref{lem1}, placed in Appendix, $F\subset\R_+$ is a set of full Lebesgue measure and each $t\in F$ is a Lebesgue point of the functions
$$I_R(t)=R^{-n}\int_{\R^n} |u(t,x)-v(t,x)| g(x/R)dx$$ for all $R>0$ and all $g(y)\in C_0(\R)$. Now we suppose that
$t_0,t_1\in F$ and pass in (\ref{3i}) to the limit as ${\nu\to\infty}$. As a result, we obtain the inequality
$$
I_R(t_1)\le I_R(t_0)+R^{-n-1}\int_\Pi \sign(u-v)(\varphi(u)-\varphi(v))\cdot\nabla_y g(x/R)\chi(t)dtdx.
$$
From this inequality and the initial conditions it follows in the limit as $F\ni t_0\to 0$ that for all $t=t_1\in F$
\begin{equation}\label{4}
I_R(t)\le I_R(0)+R^{-n-1}\int_{(0,t)\times\R^n} \sign(u-v)(\varphi(u)-\varphi(v))\cdot\nabla_y g(x/R)dtdx,
\end{equation}
where $\displaystyle I_R(0)=R^{-n}\int_{\R^n} |u_0(x)-v_0(x)| g(x/R)dx$.
Making the change $y=x/R$ in the last integral in (\ref{4}), we obtain the estimate
\begin{eqnarray}\label{5}
R^{-n-1}\left|\int_{(0,t)\times\R^n}\sign(u-v)(\varphi(u)-\varphi(v))\cdot\nabla_y g(x/R)dtdx\right|\le\nonumber\\ R^{-1}\|\varphi(u)-\varphi(v)\|_\infty\int_{(0,t)\times\R^n}|\nabla_y g|(y)dtdy\mathop{\to}_{R\to+\infty} 0.
\end{eqnarray}
Here we denote by $|v|$ the Euclidean norm of a finite-dimensional vector $v$. Further, by the properties of $g(y)$, for all $t\ge 0$
\begin{eqnarray*}
R^{-n}\int_{C_R}|u(t,x)-v(t,x)|dx\le I_R(t)\le \\ R^{-n}\int_{C_{kR}}|u(t,x)-v(t,x)|dx=k^n (kR)^{-n}\int_{C_{kR}}|u(t,x)-v(t,x)|dx,
\end{eqnarray*}
which implies that
\begin{equation}\label{6}
N_1(u(t,\cdot)-v(t,\cdot))\le \limsup_{R\to+\infty} I_R(t)\le k^n N_1(u(t,\cdot)-v(t,\cdot)).
\end{equation}
With the help of relations (\ref{5}), (\ref{6}), we derive from (\ref{4}) in the limit as $R\to+\infty$ that
$N_1(u(t,\cdot)-v(t,\cdot))\le k^n N_1(u_0-v_0)$ for all $t\in F$. To complete the proof, it only remains to notice that $k>1$ is arbitrary.
\end{proof}
\begin{remark}\label{rem1}
As was established in \cite[Corollary~7.1]{Pa6},
after possible correction on a set of null Lebesgue measure, any e.s. $u(t,x)$ is continuous on $[0,+\infty)$ as a map $t\to u(t,\cdot)\in L^1_{loc}(\R^n)$. Hence, without loss of generality, we may suppose that every e.s. satisfies the property $u(t,\cdot)\in C([0,+\infty),L^1_{loc}(\R^n))$. Then the statement of Proposition~\ref{pro1} holds for all $t>0$. The continuity assumption also allows to replace the essential limit in the initial requirement of Definition~\ref{def1} by the usual limit.
\end{remark}

Denote by $L_0^\infty(\R^n)$ the kernel of the seminorm $N_1$, that is, $L_0^\infty(\R^n)$ consists of bounded measurable functions $u=u(x)$ such that $N_1(u)=0$.
From Proposition~\ref{pro1} it readily follows the weak uniqueness property.

\begin{corollary}\label{cor1} If $u(t,x)$, $v(t,x)$ are e.s. of problem (\ref{1}), (\ref{ini}) with initial data $u_0(x), v_0(x)$, respectively, and $v_0-u_0\in L^\infty_0(\R^n)$,
then $v(t,\cdot)-u(t,\cdot)\in L^\infty_0(\R^n)$ for a.e. $t>0$.
\end{corollary}

We may consider problem (\ref{1}), (\ref{ini}) as the Cauchy problem in the quotient space $L^\infty(\R^n)/L^\infty_0(\R^n)$. In view of Corollary~\ref{cor1}, an e.s. of this problem is unique.

The aim of the present paper is investigation of the well-posedness of problem (\ref{1}), (\ref{ini}) in the class of Besicovitch almost periodic functions.
Recall, (~see \cite{Bes,Lev}~) that Besicovitch space is the closure of trigonometric polynomials, i.e., finite sums $\sum a_\lambda e^{2\pi i\lambda\cdot x}$, with ${i^2=-1}$, $\lambda\in\R^n$, in the quotient space $\B^p(\R^n)=B^p(\R^n)/B^p_0(\R^n)$, where
$$
B^p(\R^n)=\{ u\in L^p_{loc}(\R^n) \ | \ N_p(u)<+\infty \}, \ B^p_0(\R^n)=\{ u\in L^p_{loc}(\R^n) \ | \ N_p(u)=0 \}.
$$
The space $\B^p(\R^n)$ is equipped with the norm $\|u\|_p=N_p(u)$ (we identify classes in the quotient space and their representatives).
The space $\B^p(\R^n)$ is a Banach space, it is isomorphic to the completeness of the space $AP(\R^n)$ of Bohr almost periodic functions with respect to the norm $N_p$.

It is known \cite{Bes} that for each $u\in \B^p(\R^n)$ there exist the mean value
$$\dashint_{\R^n} u(x)dx\doteq\lim\limits_{R\to+\infty}R^{-n}\int_{C_R} u(x)dx$$ and the Bohr-Fourier coefficients
$$
a_\lambda=\dashint_{\R^n} u(x)e^{-2\pi i\lambda\cdot x}dx, \quad\lambda\in\R^n.
$$
The set
$$ Sp(u)=\{ \ \lambda\in\R^n \ | \ a_\lambda\not=0 \ \} $$ is called the spectrum of an almost periodic function $u$.
It is known \cite{Bes} that the spectrum $Sp(u)$ is at most countable. Denote by $M(u)$ the smallest additive subgroup of $\R^n$ containing $Sp(u)$ (notice that $M(u)$ is countable whenever it is different from the zero subgroup).

Our first result is the following

\begin{theorem}\label{th1} Let $u_0(x)\in\B^1(\R^n)\cap L^\infty(\R^n)$ be a bounded Besicovitch almost periodic function, and $u(t,x)$ be an e.s. of problem (\ref{1}), (\ref{ini}).
Then, after possible correction on a set of null measure, ${u(t,\cdot)\in C([0,+\infty),\B^1(\R^n))\cap L^\infty(\Pi)}$ and for all $t>0$  $M(u(t,\cdot))\subset M(u_0)$.
\end{theorem}

The next our result concerns the decay property of e.s. as the time ${t\to+\infty}$.

\begin{theorem}\label{th2}
Assume that
\begin{eqnarray}\label{ND}
\forall\xi\in M_0\doteq M(u_0), \xi\not=0 \ \mbox{ the functions } u\to\xi\cdot\varphi(u) \nonumber\\
\mbox{ are not affine on non-empty intervals }
\end{eqnarray}
(the linear non-degeneracy condition). Then
\begin{equation}\label{dec}
\lim_{t\to +\infty} \dashint_{\R^n}|u(t,x)-C|dx=0, \quad \mbox{ where } \quad C=\dashint_{\R^n} u_0(x)dx.
\end{equation}
Moreover, condition (\ref{ND}) is exact: if it fails, then there exists an initial function $u_0\in\B^1(\R^n)\cap L^\infty(\R^n)$ such that $Sp(u_0)\subset M_0$ and the e.s. of (\ref{1}), (\ref{ini}) does not satisfy the decay property (\ref{dec}).
\end{theorem}

Notice that for periodic initial data $u_0(x)$ with the lattice of periods $L\subset\R^n$ the group $M_0$ coincides with the dual lattice
$L'=\{ \ \xi\in\R^n \ | \ \xi\cdot x\in\Z \ \forall x\in L \ \}$, and the result of Theorem~\ref{th2} reduces to the decay property for periodic e.s. from \cite{PaAIHP}:
the linear non-degeneracy condition
\begin{eqnarray}\label{ND2}
\forall\xi\in L', \xi\not=0 \ \mbox{ the function } u\to\xi\cdot\varphi(u) \nonumber\\ \mbox{ is not affine on non-empty intervals }
\end{eqnarray}
is necessary and sufficient for the decay of every space-periodic (with the lattice of periods $L$) e.s. $u(t,x)$:
\begin{equation}\label{dec1}
\lim_{t\to +\infty} \int_{\T^n}|u(t,x)-C|dx=0, \quad C=\int_{\T^n} u_0(x)dx.
\end{equation}
Here $\T^n=\R^n/L$ is a torus (~which can be identified with the fundamental parallelepiped $$P=\{ \ x=\sum_{i=1}^n \alpha_ie_i \ | \ \alpha_i\in [0,1), \ i=1,\ldots,n \ \},$$
where $e_i$, $i=1,\ldots,n$ is a basis in $L$~), $dx$ is the normalized Lebesgue measure on $\T^n$.
Notice that for a periodic function $w(x)\in L^1(\T^n)$
$$
\dashint_{\R^n} w(x)dx=\frac{1}{|P|}\int_P w(x)dx=\int_{\T^n} w(x)dx,
$$
where $|P|$ denote the Lebesgue measure of $P$.
If the lattice of periods is not fixed and may depend on a solution, the decay property holds under the stronger assumption
\begin{eqnarray}\label{ND3}
\forall \xi\in\R^n, \xi\not=0, \mbox{ the function } u\to\xi\cdot\varphi(u) \nonumber\\
\mbox{ is not affine on non-empty intervals. }
\end{eqnarray}
This result generalizes the decay property established by G.-Q.~Chen and H.~Frid \cite{ChF}
under the conditions $\varphi(u)\in
C^2(\R,\R^n)$ and
\begin{equation}\label{ND1}
\forall (\tau,\xi)\in\R^{n+1}, (\tau,\xi)\not=0, \quad \meas\,\{ \ u\in\R \ | \
\tau+\varphi'(u)\cdot\xi=0 \ \}=0
\end{equation}
(by $\meas A$ we denote the Lebesgue measure of a measurable set $A$).
Obviously, condition (\ref{ND3}) is strictly weaker than (\ref{ND1}) even in the case of smooth flux $\varphi(u)$.
For completeness sake we confirm it by the following simple example.

\begin{example}
Let $n=1$, $A\subset [0,1]$ be a closed nowhere dense set of positive Lebesgue measure (a so-called fat Cantor set).
There exists a smooth function $\varphi(u)\in C^\infty(\R)$ such that $A=\varphi^{-1}(0)$ and $\varphi(u)$ is not affine on non-empty intervals. For instance, we may define
$\varphi(u)=e^{-\frac{1}{(u-a)(b-u)}}$ on each component interval $(a,b)$ of the complement $\R\setminus A$ (in the cases when $a=-\infty$ or when $b=+\infty$ we set
$\varphi(u)=e^{-\frac{1}{(b-u)}}$, $\varphi(u)=e^{-\frac{1}{(u-a)}}$, respectively), and set $\varphi(u)=0$ for $u\in A$. Obviously, the function $\varphi(u)$ satisfies condition (\ref{ND3}) (notice that $A$ does not contain any interval) but, since ${A\subset\{ \ u\in\R \ | \
\varphi'(u)=0 \ \}}$, condition (\ref{ND1}) fails.
\end{example}
\begin{remark}\label{rem2}
In the case when the flux vector $\varphi(u)$ satisfies the Lipschitz condition on any segment in $\R$, an e.s. $u(t,x)$ of (\ref{1}), (\ref{ini}) exhibits the property of finite speed of propagation, which implies in particular that for all $p,q\in\R^n$ and $t>0$
\begin{equation}\label{St1}
\int_{C_R}|u(t,x+p)-u(t,x+q)|dx\le\int_{C_{R+Lt}}|u_0(x+p)-u_0(x+q)|dx,
\end{equation}
where $R>0$, and $L$ is the Lipschitz constant of $\varphi(u)$ on the segment $[-M,M]$, $M=\|u\|_\infty$.

Suppose that the initial data $u_0(x)$ is a Stepanov almost periodic function (see the definition in \cite{Bes,Lev}).
It readily follows from (\ref{St1}) that $u(t,\cdot)$ is a Stepanov almost periodic function as well for all $t>0$.
The decay property (\ref{dec}) for such solutions was established in \cite{Frid} under non-degeneracy condition (\ref{ND1}) (with smooth flux) and rather restrictive assumptions on the dependence of the length of inclusion intervals for $\varepsilon$-almost periods of $u_0$ on the parameter $\varepsilon$.

In the present paper we remove all these unnecessary assumptions and prove the decay property under exact condition (\ref{ND}) for general equation (\ref{1}) with merely continuous flux and with arbitrary Besicovitch almost periodic initial data (~notice that Stepanov almost periodic functions are strictly embedded in $\B^1(\R^n)$~).
\end{remark}

\section{Initial data with finite spectrum. Reduction to the periodic case}\label{sec2}

In this section we study the case when $\displaystyle u_0(x)=\sum_{\lambda\in\Lambda}a_\lambda e^{2\pi i\lambda\cdot x}$ is a trigonometric polynomial. Here $\Lambda=Sp(u_0)\subset\R^n$ is a finite set.  We assume that $u_0(x)$ is a real function, this means that $-\Lambda=\Lambda$ and $a_{-\lambda}=\overline{a_\lambda}$ (as usual, $\bar z$ denotes the complex conjugate to $z\in\C$). The minimal additive subgroup $M_0\doteq M(u_0)$ containing $\Lambda$ is a finite generated torsion free abelian group and therefore it is a free abelian group of finite rank (see \cite{Lang}). Thus, there exists a basis $\lambda_j\in M_0$, $j=1,\ldots,m$. Every element $\lambda\in M_0$ can be uniquely represented as $\displaystyle\lambda=\lambda(\bar k)=\sum_{j=1}^m k_j\lambda_j$, $\bar k=(k_1,\ldots,k_m)\in\Z^m$. In particular, vectors $\lambda_j$, $j=1,\ldots,m$, are linearly independent over the field of rational numbers $\Q$. We introduce the finite set $J=\{ \ \bar k\in\Z^m \ | \ \lambda(\bar k)\in\Lambda \ \}$. Then we can represent the initial function as follows
$$u_0(x)=\sum_{\bar k\in J} a_{\bar k}e^{2\pi i\sum_{j=1}^m k_j\lambda_j\cdot x}, \quad a_{\bar k}\doteq a_{\lambda(\bar k)}.$$
By this representation we find that $u_0(x)=v_0(y(x))$, where
$$
v_0(y)=\sum_{\bar k\in J} a_{\bar k}e^{2\pi i\bar k\cdot y}
$$
is a periodic function in $\R^m$ with the standard lattice of periods $\Z^m$, and $y(x)$ is a linear map from $\R^n$ into $\R^m$ defined as $\displaystyle y_j=\lambda_j\cdot x=\sum_{k=1}^n\lambda_{jk}x_k$, $\lambda_{jk}$, $k=1,\ldots,n$, being coordinates of the vector $\lambda_j$, $j=1,\ldots,m$. We consider the conservation law
\begin{equation}\label{1r}
v_t+\div_y \tilde \varphi(v)=0, \quad v=v(t,y), \ t>0, \ y\in\R^m,
\end{equation}
with the flux functions $$\tilde\varphi_j(v)=\sum_{k=1}^n\lambda_{jk}\varphi_k(v)\in C(\R), \quad j=1,\ldots,m.$$
By results \cite{PaMax1,PaMax2,PaIzv} there exists a unique e.s. $v(t,y)\in L^\infty(\R_+\times\R^m)$ of the Cauchy problem for (\ref{1r}) with initial data $v_0(y)$, and this e.s. is $y$-periodic: $v(t,y+e)=v(t,y)$ a.e. in $\R_+\times\R^m$ for all $e\in\Z^m$. Moreover, in view of \cite[Corollary~7.1]{Pa6}, we may assume that $v(t,\cdot)\in C([0,+\infty), L^1(\T^m))$, where $\T^m=\R^m/\Z^m$ is an $m$-dimensional torus (~it may be identified with the fundamental cube $[0,1)^m$ in $\R^m$~).

\begin{theorem}\label{th3}
For a.e. $z\in\R^m$ the function $u(t,x)=v(t,z+y(x))$ is an e.s. of (\ref{1}), (\ref{ini}) with initial function
$v_0(z+y(x))$.
\end{theorem}
\begin{proof}
Firstly notice that the function $v(t,y)$ is an e.s. of the Cauchy problem for equation (\ref{1r}) considered
in the half-space $t>0$, $(y,x)\in\R^{m+n}$, with initial data $v_0(y)$ (we just attach the additional variables $x$). We make the non-degenerate linear change of variables $(z,x)\to (y,x)$, with $y=z+y(x)$. After this change the function $u(t,z,x)=v(t,z+y(x))$ satisfies for each $k\in\R$ the relation
\begin{eqnarray*}
|u-k|_t+\div_x[\sign(u-k)(\varphi(u)-\varphi(k))]=\\
|v-k|_t+\sum_{l=1}^n\sum_{j=1}^m [\sign(v-k)(\varphi_l(v)-\varphi_l(k))]_{y_j}\frac{\partial y_j(x)}{\partial x_l}=\\
|v-k|_t+\sum_{j=1}^m\sum_{l=1}^n [\sign(v-k)(\varphi_l(v)-\varphi_l(k))]_{y_j}\lambda_{jl}=\\
|v-k|_t+\sum_{j=1}^m [\sign(v-k)(\tilde\varphi_j(u)-\tilde\varphi_j(k))]_{y_j}\le 0 \ \mbox{ in } \D'(\R_+\times\R^{m+n}).
\end{eqnarray*}
It is clear that this function satisfies also the initial condition
$$
\lim_{t\to 0+} u(t,z,x)=u_0(z,x)\doteq v_0(z+y(x)) \ \mbox{ in } L^1_{loc}(\R^{m+n}).
$$
Thus, $u(t,z,x)$ is an e.s. of the modified problem (\ref{1}), (\ref{ini}) in the extended domain $\R_+\times\R^{m+n}$.
Define the set $E\subset\R^m$ consisting of $z\in\R^m$ such that $(t,z,x)$ is a Lebesgue point of $u(t,z,x)$ for almost all $(t,x)\in\Pi$. Then the set $E$ has full Lebesgue measure in $\R^m$. We demonstrate that for fixed $z_0\in E$ the function $u(t,z_0,x)$ is an e.s. of (\ref{1}), (\ref{ini}) with initial function $u_0(z_0,x)$. For that, we have to verify integral relation (\ref{enint}). Let us take a test function $f(t,x)\in C_0^1(\bar\Pi)$, $f(t,x)\ge 0$, and set
$f_\nu(t,z,x)=f(t,x)g_\nu(z-z_0)$, where the sequence $g_\nu(y)=\prod_{j=1}^m \delta_\nu(y_j)$, $\nu\in\N$, is an approximate unit in $\R^m$ (~so that it weakly converges as $\nu\to\infty$ to the Dirac $\delta$-function in $\D'(\R^m)$~), the functions $\delta_\nu(s)$ were defined in the proof of Proposition~\ref{pro1} above. Obviously, $f=f_\nu(t,z,x)\in C_0^1([0,+\infty)\times\R^{m+n})$, $f_\nu(t,z,x)\ge 0$. Taking $f=f_\nu(t,z,x)$ in relation (\ref{enint}) for the e.s. $u(t,z,x)$, we obtain that for each $k\in\R$
\begin{eqnarray}\label{7}
\int_{\R^m}\left(\int_{\R^n}|u_0(z,x)-k|f(0,x)dx\right)g_\nu(z-z_0)dz+ \nonumber\\
\int_{\R^m}\left(\int_{\Pi} [|u(t,z,x)-k|f_t(t,x) + \sign(u(t,z,x)-k)\times\right. \nonumber\\ \left.\vphantom{\int_{\R^m}}(\varphi(u(t,z,x))-\varphi(k))\cdot\nabla_x f(t,x)]dtdx\right)g_\nu(z-z_0)dz\ge 0.
\end{eqnarray}
Let $M=\|u\|_\infty$,
$$
\omega(\sigma)=\sup_{u,v\in [-M,M], |u-v|<\sigma} |\varphi(u)-\varphi(v)|
$$
be the continuity modulus of the vector $\varphi(u)$ on the segment $[-M,M]$. Then, as is easy to verify,
for all $u,v\in [-M,M]$, $k\in\R$
\begin{eqnarray}\label{mod1}
|\sign(u-k)(\varphi(u)-\varphi(k))-\sign(v-k)(\varphi(v)-\varphi(k))|\le\nonumber\\ 2\omega(|u-v|)\le
2\frac{\omega(\varepsilon)}{\varepsilon}|u-v|+2\omega(\varepsilon) \quad \forall\varepsilon>0.
\end{eqnarray}
Indeed, since the function $\omega(\sigma)$ is not decreasing and sub-additive, then for every $r\ge 0$
$$
\omega(r)\le\omega((m+1)\varepsilon)\le (m+1)\omega(\varepsilon)\le\frac{\omega(\varepsilon)}{\varepsilon}r+\omega(\varepsilon),
$$
where integer $m\ge 0$ satisfies the requirement $m\varepsilon\le r<(m+1)\varepsilon$.
Since ${\omega(\varepsilon)\to 0}$ as $\varepsilon\to 0$, it easily follows from (\ref{mod1}) that any Lebesgue point of
$u(t,z,x)$ is also a Lebesgue point of the vector-functions
$\sign(u(t,z,x)-k)(\varphi(u(t,z,x))-\varphi(k))$, as well as the functions $|u(t,z,x)-k|$.

Then, in view of Lemma~\ref{lem1} in Appendix, $z_0\in E$ is a Lebesgue point of the functions
\begin{eqnarray*}
I_0(z)\doteq\int_{\R^n}|u_0(z,x)-k|f(0,x)dx, \quad I(z)\doteq\int_{\Pi} [|u(t,z,x)-k|f_t(t,x)+ \\ \sign(u(t,z,x)-k)(\varphi(u(t,z,x))-\varphi(k))\cdot\nabla_x f(t,x)]dtdx
\end{eqnarray*}
(the function $I_0(z)$ is even continuous).
Therefore, it follows from (\ref{7}) in the limit as $\nu\to\infty$ that $I_0(z_0)+I(z_0)\ge 0$, i.e.,
\begin{eqnarray*}
\int_{\R^n}|u_0(z_0,x)-k|f(0,x)dx+\int_{\Pi} [|u(t,z_0,x)-k|f_t(t,x)+ \\ \sign(u(t,z_0,x)-k)(\varphi(u(t,z_0,x))-\varphi(k))\cdot\nabla_x f(t,x)]dtdx\ge 0
\end{eqnarray*}
for each $k\in\R$ and all nonnegative test functions $f(t,x)\in C_0^1(\bar\Pi)$. Hence, the function
$u=u(t,z_0,x)$ satisfies (\ref{enint}) with $u_0=u_0(z_0,x)$ and therefore it is an e.s. of (\ref{1}), (\ref{ini}) with initial function $u_0=v_0(z_0+y(x))$ for all $z_0$ from the set $E$ of full measure. The proof is complete.
\end{proof}

The additive group $\R^n$ acts on the torus $\T^m=\R^m/\Z^m$ by the shift transformations $S_xz=z+y(x)$. By linearity of $y(x)$, $S_{x_1+x_2}=S_{x_1}S_{x_2}$, so that $S_x$, $x\in\R^m$ is an $m$-parametric group of measure preserving transformations of $\T^m$. We demonstrate that this action is ergodic. Let $A\subset\T^m$ be a measurable invariant set, and $\displaystyle \chi_A(y)=\sum_{\bar k\in\Z^m}a_{\bar k} e^{2\pi i\bar k\cdot y}$
be the Fourier series for the indicator function $\chi_A(y)$ of the set $A$. Since this function is invariant, then
$\chi_A(z)=\chi_A(z+y(x))$ for all $x\in\R^n$, and
\begin{eqnarray}\label{8}
a_{\bar k}=\int_{\T^m}\chi_A(z)e^{-2\pi i\bar k\cdot z}dz=\int_{\T^m}\chi_A(z+y(x))e^{-2\pi i\bar k\cdot z}dz=\nonumber\\
\int_{\T^m}\chi_A(y)e^{-2\pi i\bar k\cdot (y-y(x))}dy=e^{2\pi i\bar k\cdot y(x)}\int_{\T^m}\chi_A(y)e^{-2\pi i\bar k\cdot y}dy=e^{2\pi i\bar k\cdot y(x)}a_{\bar k}.
\end{eqnarray}
Assume that $a_{\bar k}\not=0$. Then it follows from (\ref{8})
that $e^{2\pi i\bar k\cdot y(x)}=1$ for all $x\in\R^n$, which readily implies the relation $\displaystyle \sum_{j=1}^mk_j\lambda_j=0$, where $k_j$, $j=1,\ldots,m$, are the coordinates of $\bar k$. Since the vectors $\lambda_j$, $j=1,\ldots,m$, form a basis in $M_0$, we derive that all
$k_j=0$, that is, $\bar k=0$. Thus, only one Fourier coefficient $a_0$ may be different from zero. This yields
$\chi_A(y)\equiv c=\const$. It is clear that $c=0$ or $c=1$, which means that $m(A)=0$ or $m(A)=1$, where $m=dy$ is the Lebesgue measure on the torus $\T^m$. Thus, the action $S_xz=z+y(x)$ of the group $\R^n$ on the torus $\T^m$ is ergodic. By the variant of Birkhoff individual ergodic theorem \cite[Chapter VIII]{Danf} for each $w(y)\in L^1(\T^m)$ for almost all $z\in\T^m$ there exists the mean value
\begin{equation}\label{erg}
\dashint_{\R^n}w(z+y(x))dx=\int_{\T^m} w(y)dy.
\end{equation}
Moreover, if $w(y)\in C(\T^m)$, then (\ref{erg}) holds for all $z\in\T^m$ (this follows from uniform continuity of $w(z+y(x))$ with respect to the variables $z\in\R^m$) and $w^z(x)\doteq w(z+y(x))$ are Bohr almost periodic functions for all $z\in\T^m$.

\medskip
Now we are ready to prove our Theorems~\ref{th1},~\ref{th2} in the case when $u_0(x)$ has finite spectrum. Recall that $M_0$ is the minimal additive subgroup of $\R^n$ containing $Sp(u_0)$.

\begin{theorem}\label{th5}
Suppose that $u(t,x)$ is an e.s. of (\ref{1}), (\ref{ini}) with initial function $u_0(x)$ being a trigonometric polynomial.
Then $u(t,x)\in C([0,+\infty),\B^1(\R^n))\cap L^\infty(\Pi)$ and $Sp(u(t,\cdot))\subset M_0$ for all $t>0$. Moreover, if for all $\xi\in M_0$, $\xi\not=0$, the functions $u\to\xi\cdot\varphi(u)$ are not affine on non-empty intervals, then
the decay property (\ref{dec}) holds.
\end{theorem}
\begin{proof}
As above, we introduce the periodic function $v_0(y)$ on $\R^m$, where $m$ is the rank of $M_0$, and an e.s. $v(t,y)\in C([0,+\infty),L^1(\T^m))\cap L^\infty(\R_+\times\R^m)$ of the Cauchy problem for (\ref{1r}) with initial data $v_0(y)$.
By Theorem~\ref{th3} there exists a set $E_1\subset\R^m$ of full Lebesgue measure such that for all $z\in E_1$
the function $u^z(t,x)=v(t,z+y(x))$ is an e.s. of (\ref{1}), (\ref{ini}) with initial function $u^z_0(x)=v_0(z+y(x))$. Here the linear map $y(x)$ was defined before the formulation of Theorem~\ref{th3} above.
Let
$$
v_r(t,y)=\int_{\T^m} v(t,z)\Phi_r(y-z)dz
$$
be the Fej\'er approximations of $v(t,y)$, where
$$
\Phi_r(z)=\sum_{\bar k\in\Z^m, |\bar k|_\infty\le r} \prod_{j=1}^m\left(1-\frac{|k_j|}{r}\right) e^{2\pi i \bar k\cdot z}=
r^{-m}\prod_{j=1}^m \frac{\sin^2 \pi rz_j}{\sin^2 \pi z_j}
$$
are the Fej\'er kernels. Then $$v_r(t,y)=\sum_{\bar k\in\Z^m, |\bar k|_\infty\le r} a_{r\bar k}(t)e^{2\pi i\bar k\cdot y}$$ are trigonometric polynomials with respect to the variables $y$, and $v_r(t,\cdot)\to v(t,\cdot)$ as $r\to \infty$ in $L^1(\T^m)$ for all $t\ge 0$.
In view of (\ref{erg}) with $w(y)=|v(t,y)-v_r(t,y)|$ for all $t>0$ there exists a set $E(t)\subset\R^m$ of full Lebesgue measure consisting of $z\in\R^m$ such that the relation
\begin{eqnarray}\label{10a}
N_1(u^z(t,\cdot)-u_r^z(t,\cdot))=\dashint_{\R^n}|v(t,z+y(x))-v_r(t,z+y(x))|dx=\nonumber\\ \int_{\T^m}|v(t,y)-v_r(t,y)|dy
\end{eqnarray}
holds for all $r\in\N$, where $u_r=u_r^z(t,x)=v_r(t,z+y(x))$, $u=u^z(t,x)=v(t,z+y(x))$. It is clear that the set $G$ of pairs $(t,z)\in [0,+\infty)\times\R^m$, which satisfy for all $r\in\N$ the equality (\ref{10a}),
is measurable. We use here the fact that, due to the estimates
\begin{eqnarray*}
\left(\frac{N}{N+1}\right)^n N^{-n}\int_{C_N}w(x)dx\le R^{-n}\int_{C_R} w(x)dx\le \\
\left(\frac{N+1}{N}\right)^n (N+1)^{-n}\int_{C_{N+1}}w(x)dx, \quad \forall N\in\N, \ N\le R<N+1,
\end{eqnarray*}
which hold for every nonnegative function $w(x)\in L^1_{loc}(\R^n)$, existence of the mean value
$\displaystyle\dashint_{\R^n}w(x)dx$ is equivalent to existence of the limit $$\lim_{N\to\infty} N^{-n}\int_{C_N}w(x)dx=\dashint_{\R^n}w(x)dx.$$
Since all sections
$$G_t=\{ \ z\in\R^m \ | \ (t,z)\in G \ \}=E(t)$$ of the set $G$ have full measure, it has full Lebesgue measure as well, by Fubini's theorem. By Fubini's theorem again, there exists a set $E_2\subset\R^m$ of full measure such that the sections $$F(z)=G^z=\{ \ t\ge 0 \ | \ (t,z)\in G \ \}$$ are sets of full Lebesgue measure in $[0,+\infty)$ for each $z\in E_2$.

Now we choose a sequence $z_l\in E_1\cap E_2$ converging to zero as $l\to\infty$. By Proposition~\ref{pro1} for a.e. $t>0$
\begin{equation}\label{con}
I_1(u^{z_l}(t,x)-u(t,x))\le I_1(u^{z_l}_0(x)-u_0(x)).
\end{equation}
We denote by $F_{1l}\subset [0,+\infty)$ the set of $t$, which satisfy (\ref{con}). Let $F_{2l}=F(z_l)\subset [0,+\infty)$ be the sets of full measures, consisting of
$t\ge 0$ such that equality (\ref{10a}) holds with $z=z_l$ for all $r\in\N$. Observe that in view of the relation $\displaystyle v_r(t,\cdot)\mathop{\to}_{r\to\infty} v(t,y)$ in $L^1(\T^m)$, it follows from (\ref{10a}) that
$u^{z_l}(t,\cdot)\in\B^1(\R^n)$ and $Sp(u^{z_l}(t,\cdot))\subset M_0$ for $t\in F_{2l}$.
We set $\displaystyle F=\bigcap_{l\in\N} (F_{1l}\cap F_{2l})$. Then $F\subset [0,+\infty)$ is a set of full measure, and for each $t\in F$ \ $u^{z_l}(t,x)\in\B^1(\R^n)$ for all $l\in\N$. Observe that
$$
I_1(u^{z_l}_0(x)-u_0(x))\le\sum_{\bar k\in J}|a_{\bar k}||e^{2\pi i\bar k\cdot z_l}-1|\mathop{\to}_{l\to\infty} 0,
$$
where $J\subset\Z^m$ is a finite set (~the spectrum of $v_0(y)$~). Now it follows from (\ref{con}) that
$I_1(u^{z_l}(t,x)-u(t,x))\to 0$ as $l\to\infty$.
This implies that the limit function $u(t,\cdot)$
belongs to the Besicovitch space $\B^1(\R^n)$. Besides, since $Sp(u^{z_l}(t,\cdot))\subset M_0$ for all $l\in\N$,
then in the limit as $l\to\infty$ we obtain the required inclusion $Sp(u(t,\cdot))\subset M_0$.

Let us demonstrate continuity of the map $t\to u(t,\cdot)\in\B^1(\R^n)$. Suppose that $t,t'\in F$. Obviously, for each $r\in\N$
$$
\dashint_{R^n}|v_r(t',z+y(x))-v_r(t,z+y(x))|dx=\nonumber\\ \int_{\T^m}|v_r(t',y)-v_r(t,y)|dy,
$$
This relation in the limit as $r\to\infty$ implies, with the help of (\ref{10a}), that for all $l\in\N$
\begin{eqnarray}\label{cont}
N_1(|u^{z_l}(t',\cdot)-u^{z_l}(t,\cdot)|)=\dashint_{\R^n} |v(t',z_l+y(x))-v(t,z_l+y(x))|dx=\nonumber\\  \int_{\T^m}|v(t',y)-v(t,y)|dy.
\end{eqnarray}
Passing to the limit in (\ref{cont}) as $l\to\infty$ and taking into account (\ref{con}), we arrive at the relation
$$
N_1(|u(t',\cdot)-u(t,\cdot)|)=\int_{\T^m}|v(t',y)-v(t,y)|dy\mathop{\to}_{t'\to t} 0.
$$
This relation shows that the map $t\to u(t,\cdot)$ is uniformly continuous in $\B^1(\R^n)$ on ${F\cap [0,T]}$ for arbitrary $T>0$. Therefore, it can be uniquely extended as a continuous map $t\to u(t,\cdot)\in\B^1(\R^n)$ on the whole half-line $[0,+\infty)$. We see that, after possible correction of $u(t,\cdot)$ on a set of null measure, this map is continuous in $\B^1(\R^n)$. Notice also that $u(0,x)=u_0(x)$ (~this readily follows from the fact that $0\in F$~).
Since $u(t,x)$ is bounded, we conclude that
$u(t,x)\in C([0,+\infty),\B^1(\R^n))\cap L^\infty(\Pi)$. The first statement is proved.

\medskip
To prove the decay property, we notice that for every $\bar k=(k_1,\ldots,k_m)\in\Z^m$
$$
\bar k\cdot\tilde\varphi(u)=\sum_{j=1}^m\sum_{k=1}^n k_j\lambda_{jk}\varphi_k(u)=
\lambda(\bar k)\cdot\varphi(u),
$$
where $\displaystyle\lambda(\bar k)=\sum_{j=1}^m k_j\lambda_j\in M_0$. It now follows from our assumption (\ref{ND}) that the functions $u\to\bar k\cdot\tilde\varphi(u)$ are not affine on non-empty intervals. This means that non-degeneracy condition (\ref{ND2}) is satisfied (with $L'=L=\Z^m$).
As follows from the result of \cite{PaAIHP}, the periodic e.s. $v(t,y)$ satisfies the decay relation:
\begin{eqnarray}\label{11}
\lim_{t\to+\infty}\int_{\T^m}|v(t,y)-C|dy=0,  \\ \nonumber \mbox{ where }
C=\int_{\T^m} v_0(y)dy=\dashint_{\R^n}u_0(x)dx.
\end{eqnarray}
Now we choose a vanishing sequence $z_l\in E_1\cap E_2$ and the set $F\subset [0,+\infty)$ of full measure as in the first part of our proof.
By (\ref{erg}) with $w(y)=|v_r(t,y)-C|\in C(\T^m)$
$$
\dashint_{\R^n}|u_r^{z_l}(t,x)-C|dx=\dashint_{\R^n}|v_r(t,z_l+y(x))-C|dx=\int_{\T^m}|v_r(t,y)-C|dy.
$$
Passing in this equality to the limit, first as $r\to\infty$ and then as $l\to\infty$, and taking into account relations (\ref{10a}), (\ref{con}), we obtain that for all $t\in F$
$$
\dashint_{\R^n}|u(t,x)-C|dx=\int_{\T^m}|v(t,y)-C|dy.
$$
Since both parts of this equality are continuous with respect to $t\in [0,+\infty)$, it remains valid for all $t>0$.
In view of (\ref{11}), this implies the desired decay property (\ref{dec}):
$$
\lim_{t\to+\infty}\dashint_{\R^n}|u(t,x)-C|dx=0.
$$
The proof is complete.
\end{proof}

\section{The general case. Proof of Theorems~\ref{th1},~\ref{th2}}

Assume that ${u_0\in\B^1(\R^n)\cap L^\infty(\R^n)}$ is an arbitrary bounded Besicovitch almost periodic function, and
$u(t,x)\in L^\infty(\Pi)$ is an e.s. of (\ref{1}), (\ref{ini}).

Denote by $M(u_0)$ the minimal additive subgroup of $\R^n$ containing $Sp(u_0)$. Let $u_{0l}$ be the sequence of  trigonometric polynomials such that $u_{0l}\to u_0$ as $l\to\infty$ in $\B^1(\R^n)$, and $Sp(u_{0l})\subset M(u_0)$
(~for instance we may choose the Bochner-Fej\'er trigonometric polynomials for $u_0$, see \cite{Bes}~). We denote by $u_l(t,x)$ an e.s. of problem (\ref{1}), (\ref{ini}) with initial function $u_{0l}(x)$. By Proposition~\ref{pro1}, there exists a set $F\subset\R_+$ of full Lebesgue measure such that for all $t\in F$ for every $l\in\N$
\begin{equation}\label{13}
N_1(u_l(t,\cdot)-u(t,\cdot))\le N_1(u_{0l}-u_0)\mathop{\to}_{l\to\infty} 0.
\end{equation}
Since the function $u_{0l}(x)$ has finite spectrum, then by Theorem~\ref{th5} we see that $u_l(t,x)\in C([0,+\infty),\B^1(\R^n))$ and $Sp(u_l(t,\cdot))\subset M(u_0)$ for all $l\in\N$. In view of uniform limit relation (\ref{13}) we conclude, in the limit as $l\to\infty$, that $u(t,\cdot)$ is uniformly continuous map on $F\cap [0,T]$ into $\B^1(\R^n)$ for every $T>0$, and $Sp(u(t,\cdot)\subset M(u_0)$ for all $t\in F$. This allows to extend $u(t,x)$ as a function $C([0,+\infty),\B^1(\R^n))$ on the whole half-line $[0,+\infty)$. By continuity, we conclude that $Sp(u(t,\cdot))\subset M(u_0)$ for $t>0$. The proof of Theorem~\ref{th1} is complete.

\medskip
Denote by $M_l$ the minimal additive subgroup of $\R^n$ contained $Sp(u_{0l})$. Then $M_l\subset M(u_0)$.
Assume that the linear non-degeneracy condition is satisfied. Then for each $l\in\N$ and for all $\xi\in M_l$, $\xi\not=0$,
the function $u\to\xi\cdot\varphi(u)$ is not affine on non-empty intervals. By Theorem~\ref{th5} again
\begin{equation}\label{14}
\lim_{t\to+\infty}\dashint_{\R^n}|u_l(t,x)-C_l|dx=0,
\end{equation}

where
\begin{equation}\label{15}
C_l=\dashint_{\R^n} u_{0l}(x)dx\mathop{\to}_{l\to\infty} C=\dashint_{\R^n} u_0(x)dx.
\end{equation}
In view of (\ref{13}), (\ref{15}), it readily follows from (\ref{14}) in the limit as $l\to\infty$ that
$$
\lim_{t\to+\infty}\dashint_{\R^n}|u(t,x)-C|dx=0,
$$
and the decay property holds.

\medskip
Conversely, assume that non-degeneracy condition (\ref{ND}) fails. Then there exists a nonzero vector $\xi\in M_0$ such that the function
$\xi\cdot\varphi(u)=\tau u+c$ on some segment $[a,b]$, where $\tau,c\in\R$. Then, as is easy to verify, the function
$$
u(t,x)=\frac{a+b}{2}+\frac{b-a}{2}\sin(2\pi(\xi\cdot x-\tau t))
$$
is an e.s. of (\ref{1}), (\ref{ini}) with the periodic initial data $$u(0,x)=\frac{a+b}{2}+\frac{b-a}{2}\sin(2\pi(\xi\cdot x))$$ such that $Sp(u_0)=\{-\xi,\xi\}\subset M_0$. Obviously, the e.s. $u(t,x)$ does not satisfy the decay property.

\begin{remark}
The results of this paper can be easily extended to the case of unbounded almost periodic solutions
$u(t,x)\in C([0,+\infty),\B^1(\R^n))$. Indeed, if $u_0(x)\in\B_1(\R^n)$ then there exists a sequence $u_l(x)\in\B^1(\R^n)\cap L^\infty(\R^n)$ (~for instance, the sequence of Bochner-Fej\'er approximations for $u_0$~) such that $u_l(x)\to u(x)$ as $l\to\infty$ and that $Sp(u_l)\subset M(u_0)$. By Proposition~\ref{pro1}
for all $t>0$, $k,l\in\N$
$$
N_1(u_k(t,\cdot)-u_l(t,\cdot))\le N_1(u_k-u_l)\mathop{\to}_{k,l\to\infty} 0.
$$
This implies that $u_l(t,x)$, $l\in\N$, is a Cauchy sequence in $C([0,+\infty),\B^1(\R^n))$. Since this space is complete, we claim that $u_l(t,x)$ converges as $l\to\infty$ to a function $u(t,x)$ in $C([0,+\infty),\B^1(\R^n))$.
Obviously, this limit function does not depend on the choice of a sequence $u_l(x)\in\B^1(\R^n)\cap L^\infty(\R^n)$.
It is natural to call the function $u(t,x)\in C([0,+\infty),\B^1(\R^n))$ a renormalized solution to problem (\ref{1}), (\ref{ini}), cf. the notion of renormalized solution $u(t,x)\in C([0,+\infty),L^1(\R^n))$ defined in
\cite{BCW}, see also further results in \cite{LyPa1,LyPa2}. It readily follows from the decay property for bounded e.s. $u_l$ in the limit as $u_l\to u$ that any renormalized solution also satisfies the decay property (\ref{dec}) under non-degeneracy condition (\ref{ND}).
\end{remark}

\section{Appendix: The technical lemma}

\begin{lemma}\label{lem1}
Suppose that $u(x,y)\in L^\infty(\R^n\times\R^m)$,
$$
E=\{ \ x\in\R^n \ | \ (x,y) \mbox{ is a Lebesgue point of } u(x,y) \mbox{ for a.e. } y\in\R^m \ \}.
$$
Then $E$ is a set of full measure and $x\in E$ is a common Lebesgue point of the functions $\displaystyle I(x)=\int_{\R^m} u(x,y)\rho(y)dy$, where $\rho(y)\in L^1(\R^m)$.
\end{lemma}

\begin{proof}
Since the set of Lebesgue points of $u(x,y)$ has full Lebesgue measure, then by Fubini's theorem $E\subset\R^n$ is a set of full measure.

Let $w(s)$ be the indicator function of the segment $[-1/2,1/2]$. We define sequences (approximate unities) $w_\nu(s)=\nu w(\nu s)$,
$\displaystyle w_\nu^k(z)=\prod_{j=1}^kw_\nu(z_j)$, $z=(z_1,\ldots,z_k)\in\R^k$, $\nu\in\N$.
Remark that the sequence of averaged functions
$$
\rho_\nu(y)=\rho*w_\nu^m(y)=\int_{\R^m}\rho(z)w_\nu^m(y-z)dz\mathop{\to}_{\nu\to\infty} \rho(y) \ \mbox{ in } L^1(\R^m).
$$
This relation implies that
\begin{equation}\label{l1}
I_\nu(x)=\int_{\R^m} u(x,y)\rho_\nu(y)dy\mathop{\to}_{\nu\to\infty} I(x)
\end{equation}
uniformly on $\R^n$. Notice also that
\begin{equation}\label{l2}
I_\nu(x)=\int_{\R^m}u_\nu(x,y)\rho(y)dy,
\end{equation}
where
$$
u_\nu(x,y)=\int_{\R^m} u(x,z)w_\nu^m(z-y)dz.
$$
Therefore, for $x_0\in E$
$$
I_\nu(x)-I(x_0)=\int_{\R^m}\left(\int_{\R^m}(u(x,z)-u(x_0,y))w_\nu^m(z-y)dz\right)\rho(y)dy,
$$
which implies
\begin{eqnarray}\label{l3}
\int_{\R^n}|I_\nu(x)-I(x_0)|w_\nu^n(x-x_0)dx\le\nonumber\\ \int_{\R^m}\left(\int_{\R^n\times\R^m}|u(x,z)-u(x_0,y)|w_\nu^n(x-x_0)w_\nu^m(z-y)dxdz\right)\rho(y)dy.
\end{eqnarray}
Since $(x_0,y)$ is a Lebesgue point of $u(x,z)$ for a.e. $y\in\R^m$, then
$$
\int_{\R^n\times\R^m}|u(x,z)-u(x_0,y)|w_\nu^n(x-x_0)w_\nu^m(z-y)dxdz\mathop{\to}_{\nu\to\infty} 0
$$
for a.e. $y\in\R^m$, and by Lebesgue's dominated convergence theorem the right-hand side of (\ref{l3}) converges to zero as $\nu\to\infty$.
Thus,
\begin{equation}\label{l4}
\int_{\R^n}|I_\nu(x)-I(x_0)|w_\nu^n(x-x_0)dx\mathop{\to}_{\nu\to\infty} 0.
\end{equation}
In view of (\ref{l1}) we see that
$$
\int_{\R^n}|I_\nu(x)-I(x)|w_\nu^n(x-x_0)dx\mathop{\to}_{\nu\to\infty} 0,
$$
which together with (\ref{l4}) yields
$$
\int_{\R^n}|I(x)-I(x_0)|w_\nu^n(x-x_0)dx\mathop{\to}_{\nu\to\infty} 0
$$
and shows that $x_0$ is a Lebesgue point of $I(x)$.
\end{proof}

{\bf Acknowledgement.}
This research was carried out with the financial support of the Russian Foundation for Basic Research (grant no. 12-01-00230) and the Ministry of Education and Science of Russian Federation (in the framework of state task).


\begin{thebibliography}{99}
\parskip=-2pt
\bibitem{BCW}
\textsc{B\'enilan~Ph., Carrillo~J., Wittbold~P.}:
Renormalized entropy solutions of scalar conservation laws.
\textit{Ann. Scuola Norm. Sup. Pisa Cl. Sci.} \textbf{29}, 313--327 (2000)
\bibitem{Bes}
\textsc{Besicovitch~A.\,S.}: \textit{Almost Periodic Functions}. Cambridge University Press, 1932
\bibitem{ChF}
\textsc{Chen~G.-Q., Frid~H.}:
Decay of entropy solutions of nonlinear conservation laws.
\textit{Arch. Rational Mech. Anal.} \textbf{146}:2, 95--127 (1999)
\bibitem{Frid}
\textsc{Frid~H.}: Decay of Almost Periodic Solutions of Conservation Laws. \textit{ Arch. Rational Mech. Anal.} \textbf{161}, 43–-64 (2002)
\bibitem{Danf}
\textsc{Danford~N., Schwartz~J.\,T.}: \textit{Linear Operators. General Theory (Part I)}. Interscience Publishers, 1958
\bibitem{Kr}
\textsc{Kruzhkov~S.\,N.}: First order quasilinear equations in several independent variables.
\textit{Mat. Sb.} \textbf{81}, 228--255 (1970); Engl. transl. in \textit{Math. USSR Sb.} \textbf{10}, 217--243  (1970)
\bibitem{KrPa1}
\textsc{Kruzhkov~S.\,N., Panov~E.\,Yu.}: First-order conservative quasilinear laws with an infinite
domain of dependence on the initial data. \textit{Dokl. Akad. Nauk SSSR} \textbf{314}, 79--84 (1990);
Engl. transl. in \textit{Soviet Math. Dokl.} \textbf{42}, 316--321 (1991)
\bibitem{KrPa2}
\textsc{Kruzhkov~S.\,N., Panov~E.\,Yu.}: Osgood's type conditions for uniqueness of entropy solutions
to Cauchy problem for quasilinear conservation laws of the first order. \textit{Ann. Univ. Ferrara
Sez. VII (N.S.)} \textbf{40}, 31--54 (1994)
\bibitem{Lang}
\textsc{Lang~S.}: \textit{Algebra (Revised 3rd ed.)}. New York: Springer-Verlag, 2002.
\bibitem{Lev}
\textsc{Levitan~B.\,M.}: \textit{Almost periodic functions}. Gostekhizdat: Moscow, 1953 (Russian)
\bibitem{LyPa1}
\textsc{Lysuho~P.\,V., Panov~E.\,Yu.}: Renormalized entropy solutions to the Cauchy problem
for a quasilinear first order equation. \textit{J. Math. Sci.} \textbf{172}, 1--23 (2011)
\bibitem{LyPa2}
\textsc{Lysuho~P.\,V., Panov~E.\,Yu.}: Renormalized entropy solutions to the Cauchy problem for first
order quasilinear conservation laws in the class of periodic functions. \textit{J. Math. Sci.} \textbf{177}, 27--49 (2011)
\bibitem{PaMV}
\textsc{Panov~E.\,Yu.}: On measure-valued solutions of the Cauchy problem for a first-order quasilinear equation.
\textit{Izvestiya RAN: Ser. Mat.} \textbf{60}:2, 107--148 (1996); Engl. transl. in \textit{Izvestiya: Mathematics} \textbf{60}:2, 335--377 (1996)

\bibitem{PaMax1}
\textsc{Panov~E.\,Yu.}: A remark on the  theory of generalized entropy sub- and supersolutions of the
Cauchy problem for  a  first-order quasilinear equation. \textit{Differ. Uravn.} \textbf{37}, 252--259 (2001); Engl. transl. in \textit{Differ. Equ.} \textbf{37}, 272--280 (2001).
\bibitem{PaMax2}
\textsc{Panov~E.\,Yu.}: Maximum and minimum generalized entropy solutions to the Cauchy problem for a
first-order quasilinear equation. \textit{Mat. Sb.} \textbf{193}, 95--112 (2002); Engl.
transl. in \textit{Russian Acad. Sci. Sb. Math.} \textbf{193}, 727-–743 (2002)
\bibitem{PaIzv}
\textsc{Panov~E.\,Yu.}: On generalized entropy solutions of the Cauchy problem for a first order quasilinear equation in the class of locally summable functions. \textit{Izvestiya RAN: Ser. Mat.} \textbf{66}:6, 91--136 (2002); Engl. transl. in \textit{Izvestiya: Mathematics} \textbf{66}:6, 1171--1218 (2002)
\bibitem{Pa6}
\textsc{Panov~E.\,Yu.}: Existence of strong traces for generalized solutions of multidimensional
scalar conservation laws. \textit{J. Hyperbolic Differ. Equ.} \textbf{2}, 885--908 (2005)
\bibitem{PaAIHP}
\textsc{Panov~E.\,Yu.}: On decay of periodic entropy solutions to a scalar conservation law. \textit{Annales de l'Institut Henri Poincar\'e (C) Analyse Non Lin\'eaire} \textbf{30}, 997--1007 (2013)
\end{thebibliography}
\end{document}